\documentclass[preprint, 10pt, english]{elsarticle}
\usepackage{amsthm}
\usepackage{amsmath}
\usepackage{latexsym, amssymb}
\usepackage{txfonts}
\usepackage{color}
\usepackage{babel}
\usepackage[all]{xy}
\usepackage[capitalise]{cleveref}

\newtheorem{thm}{Theorem}[section] 

\newtheorem{cor}[thm]{Corollary}

\newtheorem{lem}[thm]{Lemma}

\newtheorem{ques}[thm]{Question}

\theoremstyle{definition}
\newtheorem{rem}[thm]{Remark}

\newcommand\operA[2]{{\if!#2!\operatorname{#1}\else{\operatorname{#1}_{#2}^{\phantom{I}}}\fi}} 

\newcommand{\Trace}[1][]{\if!#1!\operatorname{Tr}\else{\operatorname{Tr}_{#1}^{\phantom{I}}}\fi} 

\long\def\forget#1\forgotten{{}} %

\def\({\left(}
\def\){\right)}

\newcommand\LAY[3][]{{\begin{array}{c}\mbox{#2} \if#1!{}\else{+}\fi \\ \mbox{#3}\end{array}}}

\makeatletter

\def\ps@pprintTitle{%
 \let\@oddhead\@empty
 \let\@evenhead\@empty
 \def\@oddfoot{}%
 \let\@evenfoot\@oddfoot}

\newcommand{\bigperp}{%
  \mathop{\mathpalette\bigp@rp\relax}%
  \displaylimits
}

\newcommand{\bigp@rp}[2]{%
  \vcenter{
    \m@th\hbox{\scalebox{\ifx#1\displaystyle2.1\else1.5\fi}{$#1\perp$}}
  }%
}
\makeatother

\newcommand{\qf}[1]{\mbox{$\langle #1\rangle $}}

\newcommand{\Om}{\Omega}

\newcommand{\wg}{\wedge}
\newcommand{\bwg}{\bigwedge}

\renewcommand{\geq}{\geqslant}
\renewcommand{\leq}{\leqslant}

\DeclareMathOperator{\coker}{coker}

\newif\iffurther
\furtherfalse

\journal{??}

\begin{document}
\begin{frontmatter}

\title{Differential Forms, Linked Fields and the $u$-Invariant}

\author{Adam Chapman}
\ead{adam1chapman@yahoo.com}
\address{Department of Computer Science, Tel-Hai Academic College, Upper Galilee, 12208 Israel}
\author{Andrew Dolphin}
\ead{Andrew.Dolphin@uantwerpen.be}
\address{Department of Mathematics, Ghent University,  Ghent, Belgium}

\begin{abstract}
We associate an Albert form to any pair of cyclic algebras of prime degree $p$ over a field $F$ with $\operatorname{char}(F)=p$ which coincides with the classical Albert form when $p=2$. We prove that if every Albert form is isotropic then $H^4(F)=0$.
As a result, we obtain that if $F$ is a linked field with $\operatorname{char}(F)=2$ then its $u$-invariant is either $0,2,4$ or $8$.
\end{abstract}

\end{frontmatter}

\section{Introduction}

Given a field $F$, a quaternion algebra over $F$ is a central simple $F$-algebra of degree $2$.
The maximal subfields of  quaternion division algebras over $F$ are quadratic field extensions of $F$.
When $\operatorname{char}(F) \neq 2$, all quadratic field extensions are separable.  
When $\operatorname{char}(F)=2$, there are two types of quadratic field extensions: the separable type which is of the form $F[x : x^2+x=\alpha]$ for some $\alpha \in F \setminus \{\lambda^2+\lambda : \lambda \in F\}$, and the inseparable type which is of the form $F[\sqrt{\alpha}]$ for some $\alpha \in F^\times \setminus (F^\times)^2$.
In this case, any quaternion division algebra contains both types of field extensions, which can be seen by its 
 symbol presentation 
$$[\alpha,\beta)_{2,F}=F \langle x,y : x^2+x=\alpha, y^2=\beta, y x y^{-1}=x+1 \rangle\,.$$
If    $\operatorname{char}(F) = p$ for some prime $p>0$, we let $\wp(F)$ denote the additive subgroup $ \{\lambda^p-\lambda: \lambda \in F\}$. Then 
 we may consider cyclic division algebras over $F$ of degree $p$. Any such algebra admits a symbol presentation 
$$[\alpha,\beta)_{p,F}=F \langle x,y : x^p-x=\alpha, y^p=\beta, y x y^{-1}=x+1 \rangle\,$$
where $\alpha \in F \setminus\wp(F)$ and $\beta \in F^\times \setminus (F^\times)^p$. In particular, these algebras contain both cyclic separable   field extensions of $F$ (e.g.~$F[x]$) of degree $p$ and purely inseparable field extensions of $F$ of degree $p$ (e.g.~$F[y]$).

Two quaternion $F$-algebras are called \emph{linked} if they share a common maximal subfield. When $\operatorname{char}(F)=2$, the notion of linkage can be refined to \emph{separable linkage} and \emph{inseparable linkage} depending on the type of quadratic field extension of the center they share. Inseparable linkage implies separable linkage, but  the converse does not hold in general (see \cite{Lam:2002}). This observation was extended to Hurwitz algebras in \cite{ElduqueVilla:2005} and to quadratic Pfister forms in \cite{Faivre:thesis}.
We similarly call cyclic $p$-algebras of prime degree $p$ over a field $F$ \emph{separably  linked} (resp.~\emph{inseparably linked}) if they share a common maximal subfield that is a cyclic separable (resp.~purely inseparable) extension of $F$ of degree $p$. The above linkage result for quaternion algebras was generalized to this setting in \cite{Chapman:2015}.

A field $F$ is called \emph{linked} if every two quaternion $F$-algebras are linked. When $\operatorname{char}(F)=2$, a field $F$ is called inseparably linked  if every two quaternion $F$-algebras are inseparably linked. Note that any inseparably linked field is clearly linked. 

The $u$-invariant of a field $F$, denoted by $u(F)$, is defined to be the maximal dimension of an anisotropic nonsingular quadratic form over $F$ of finite order in $W_q F$.
Note that when $-1$ can be written as a sum of squares in $F$, and in particular when $\operatorname{char}(F)=2$, every form in $I_q F$ is of finite order.
It was proven in \cite[Main Theorem]{ElmanLam:1973} that if $F$ is a linked field with $\operatorname{char}(F) \neq 2$ then the possible values $u(F)$ can take are $0,1,2,4$ and 8.
For fields $F$ of characteristic $2$, it was shown in
 \cite[Theorem 3.1]{Baeza:1982} that $F$ is inseparably linked if and only if $u(F) \leq 4$. In particular, this means that a linked field $F$ with $u(F)=8$ is not inseparably linked. For example, the field of iterated Laurent series in two variables $\mathbb{F}_2((\alpha))((\beta))$ over $\mathbb{F}_2$ is linked by \cite[Corollary 3.5]{AJ95}, but not inseparably linked, because its $u$-invariant is $8$. There are also many examples of inseparably linked fields, such as local fields, global fields and Laurent series over perfect fields (see \cite[Section 6]{CDL}). In 
 \cite[Theorem 3.3.10]{Faivre:thesis} it was shown that if $F$ is a linked field and $I_q^4 F=0$  (see \cref{secbfqf}) then $u(F)$ is either $0,2,4$ or $8$.
We are interested in removing the assumption that $I_q^4 F=0$ from this result. 

We approach this problem from the more general setting of differential forms over fields of characteristic $p$ (see Section \ref{Diff}). We associate an Albert form to any pair of cyclic algebras of degree $p$ over a field $F$ with $\operatorname{char}(F)=p$ which coincides with the classical Albert form when $p=2$. We prove that if every Albert form is isotropic then $H^4(F)=0$. When $p=2$, this means that if $F$ is linked then $I_q^4 F=0$. Together with  \cite[Theorem 3.3.10]{Faivre:thesis}, this gives that the possible values of $u(F)$ are $0,2,4$ and $8$.

\section{Bilinear and Quadratic Pfister Forms}\label{secbfqf}

We  recall  certain  results and terminology  we use from quadratic  form theory. 
We refer to  \cite[Chapters 1 and 2]{EKM} for standard notation, basic results and  as a general reference on quadratic forms.

 Let $F$ be a field of characteristic 2.
A symmetric bilinear form over $F$ is a map $B : V \times V \rightarrow F$ satisfying $B(v,w)=B(w,v)$, $B(cv,w)=c B(v,w)$ and $B(v+w,t)=B(v,t)+B(w,t)$ for all $v,w,t \in V$ and $c \in F$ where $V$ is an $n$-dimensional $F$-vector space.
A symmetric bilinear form $B$ is degenerate if there exists a vector $v \in V\setminus\{0\}$ such that $B(v,w)=0$ for all $w \in V$.
If such a vector does not exist, we say that $B$ is nondegenerate.
Two symmetric bilinear  forms $B: V\times V \rightarrow F$ and $B' : W \times W\rightarrow F$ are isometric if there exists an isomorphism $M : V \rightarrow W$ such that $B(v,v')=B'(Mv,Mv')$ for all $v,v' \in V$.

A quadratic form over $F$ is a map $\varphi : V \rightarrow F$ 
such that $\varphi(av)= a^2\varphi(v)$ for all $a\in F$ and $v\in V$ and the map defined by  $B_\varphi(v,w)=\varphi(v+w)-\varphi(v)-\varphi(w)$ for all $v,w\in V$ is a bilinear form. The bilinear form $B_\varphi$ is called the polar form of $\varphi$ and is clearly symmetric.  
  Two quadratic forms $\varphi : V \rightarrow F$ and $\psi : W \rightarrow F$ are isometric if there exists an isomorphism $M : V \rightarrow W$ such that $\varphi(v)=\psi(Mv)$ for all $v \in V$.
We are interested in the isometry classes of quadratic forms, so when we write $\varphi=\psi$ we actually mean that they are isometric.

 We say that $\varphi$ is singular if $B_\varphi$ is degenerate, and that $\varphi$ is nonsingular if $B_\varphi$ is nondegenerate.
Every nonsingular form $\varphi$ is even dimensional and can be written as 
$$\varphi=[\alpha_1,\beta_1] \perp \dots \perp [\alpha_n,\beta_n]$$
for some $\alpha_1,\dots,\beta_n \in F$, where $[\alpha,\beta]$ denotes the two-dimensional quadratic form $\psi(x,y)=\alpha x^2+xy+\beta y^2$ and $\perp$ denotes  the orthogonal sum of quadratic forms.

We say that a quadratic form $\varphi : V \rightarrow F$ is isotropic if there exists a vector $v \in V\setminus\{0\}$ such that $\varphi(v)=0$.
If such a vector does not exist, we say that $\varphi$ is anisotropic.
The unique nonsingular two-dimensional isotropic quadratic form is $\varmathbb{H}=[0,0]$, which we call the {hyperbolic plane}.
A hyperbolic form is an orthogonal sum of hyperbolic planes.
We say that two nonsingular quadratic forms are Witt equivalent if their orthogonal sum is a hyperbolic form.

We denote by $\langle \alpha_1,\dots,\alpha_n \rangle$ the diagonal bilinear form given by  
$(x,y)\mapsto \sum_{i=1}^n \alpha_ix_iy_i$. 
Given two symmetric bilinear forms $B_1:V\times V\rightarrow F$ and $B_2:W\times W\rightarrow F$,   the tensor product of $B_1$ and $B_2$ denoted $B_1\otimes B_2$ is 
  the unique $F$-bilinear map $B_1\otimes B_2:(V\otimes_F W)\times (V\otimes_F W)\rightarrow F$ such that 
$$(B_1\otimes B_2)\left( (v_1\otimes w_1), (v_2\otimes w_2)\right) =B_1(v_1,v_2)\cdot B_2(w_1,w_2) $$
for all $w_1,w_2\in W, v_1,v_2\in V$.
A bilinear $n$-fold Pfister form over $F$ is a symmetric bilinear form isometric to 
$\langle 1,\alpha_1\rangle \otimes \dots \otimes \langle 1,\alpha_n\rangle$
for some $\alpha_1,\alpha_2,\dots,\alpha_n \in F^\times$.
We denote such a form by $\langle \langle \alpha_1,\alpha_2,\dots,\alpha_n \rangle \rangle$. By convention, the bilinear 0-fold Pfister form is $\langle 1 \rangle$.

  Let $B:V\times V\rightarrow F$ be a symmetric bilinear form over $F$ and $\varphi:W\rightarrow F$ be a quadratic form over $F$. We may define a  quadratic form $B\otimes \varphi:V\otimes_F W\rightarrow F$ determined by the rule that 
$( B\otimes \varphi) (v\otimes w)=  B(v,v) \cdot \varphi(w)$
for all $w\in W, v\in V$. We call this quadratic form  the tensor product of $B$ and $\varphi$. 
A quadratic $n$-fold Pfister form over $F$ is a tensor product of a bilinear $(n-1)$-fold Pfister form  $\langle \langle \alpha_1,\alpha_2,\dots,\alpha_{n-1} \rangle \rangle$ and a two-dimensional quadratic form  $[1,\beta]$ for some $\beta \in F$.
We denote such a form by $\langle \langle \alpha_1,\dots,\alpha_{n-1},\beta]]$. Quadratic $n$-fold Pfister forms are isotropic if and only if they are hyperbolic (see  \cite[(9.10)]{EKM}).

The Witt equivalence classes of nonsingular quadratic forms over $F$ form an abelian group, called the Witt group of $F$, with $\perp$ as the binary group operation and $\varmathbb{H}$ as the zero element. We denote this group by $I_q F$ or $I_q^1 F$.
This group is generated by scalar multiples of quadratic 1-fold Pfister forms.
Let $I_q^n F$ denote the subgroup generated by scalar multiples of quadratic $n$-fold Pfister forms over $F$.

Let  $\varphi=[\alpha_1,\beta_1] \perp \dots \perp [\alpha_n,\beta_n]$ be a nonsingular quadratic form. The Arf invariant of $\varphi$, denoted $\triangle(\varphi)$, is the class of $\alpha_1\beta_1+\cdots +\alpha_n\beta_n$ in the additive group $F/\wp (F)$ (see \cite[\S13]{EKM}).
The Arf invariant only depends on the class of the form $\varphi$ in $I_qF$.
An Albert form over a field of characteristic $2$ is a $6$-dimensional nonsingular quadratic form 
with trivial Arf invariant. 
 To any  central simple algebra isomorphic to the tensor product of two quaternion algebras  over $F$, we may associate the Witt class of the orthogonal sum of the two norm forms of the quaternion algebras  (these norm forms are $2$-fold quadratic Pfister forms). This uniquely determines a similarity class of  Albert forms. Conversely, every similarity class of Albert forms determines such a central simple algebra over $F$  (see \cite{MammoneShapiro} for more details).


\section{Differential Forms}\label{Diff}

Let $F$ be a field of characteristic $p>0$. For $a\in F$, we  denote the  extension of $F$ isomorphic to $F[T]/(T^p-T-a)$ by $F_a$. If $a\notin \wp(F)$, then this is a cyclic field extension of degree $p$ and we denote the norm map by $N_{F_a/F}:F_a\rightarrow F$. 
Otherwise $F_a$ is an \'etale extension isomorphic to $F\times \ldots \times F$ ($p$ times), and one defines a norm map by taking the determinant of the $F$-linear map given by multiplying by an element of $F_a$. We again denote this map by $N_{F_a/F}$. It is easily seen that $N_{F_a/F}$ has a non-trivial zero if and only if $F_a$ is not a field if and only if $a\notin \wp(F)$. 


The space 
$\Om^1(F)$ of absolute {differential $1$-forms} 
over $F$ is defined to be the
$F$-vector space generated by symbols $da$, $a\in F$, subject to the relations
given by additivity, $d(a+b)=da+db$, and the product rule, $d(ab)=adb+bda$.
In particular, one has $d(F^p)=0$ for $F^p=\{ a^p\,|\,a\in F\}$, and
$d\,:\,F\to \Om^1(F)$ is an $F^p$-derivation.

The space of {$n$-differentials} $\Om^n(F)$ ($n\geq 1$) is then defined by the
$n$-fold exterior power, 
$\Om^n(F):=\bwg^n(\Om^1(F))$, which is therefore an $F$-vector space generated
by {symbols} $da_1\wg\ldots\wg da_n$, $a_i\in F$. The derivation $d$ extends to an 
operator $d\,:\,\Om^n(F)\to \Om^{n+1}(F)$ by $d(a_0da_1\wg\ldots\wg da_n)=
da_0\wg da_1\wg\ldots\wg da_n$.  We put $\Om^0(F)=F$, $\Om^n(F)=0$ for $n<0$, and
$\Om(F)=\bigoplus_{n\geq 0}\Om^n(F)$,  the algebra of differential forms
over $F$ with multiplication naturally defined by
$$(a_0da_1\wg\ldots\wg da_n)(b_0db_1\wg\ldots\wg db_m)=
a_0b_0da_1\wg\ldots\wg da_n\wg db_1\wg\ldots\wg db_m\,.$$  Note that the wedge product is anti-commutative. That is $da\wg db=-db\wg da$.

There exists a well-defined group homomorphism $\Om^n(F)\to \Om^n(F)/d\Om^{n-1}(F)$, the
Artin-Schreier map $\wp$, which acts on {logarithmic} differentials as follows:
$$b\frac{da_1}{a_1}\wg\ldots\wg \frac{da_n}{a_n}\,\longmapsto\,
(b^p-b)\frac{da_1}{a_1}\wg\ldots\wg \frac{da_n}{a_n}$$
We define $H^{n+1}(F):=\coker(\wp)$.
The  connection between the groups $H^{n+1}(F)$  and quadratic forms was shown
by Kato \cite{k1}:
\begin{thm}\label{kato}  Let $F$ be a field of characteristic $2$.  Then there is an 
isomorphism $\alpha_{n,F}\,:\,H^{n+1}(F) \stackrel{\sim}{\longrightarrow} I_q^{n+1}(F)/I_q^{n+2}(F)$ 
defined on generators as follows:
$$b\frac{da_1}{a_1}\wg\ldots\wg \frac{da_n}{a_n}
\longmapsto   \langle \langle a_1,\dots,a_n,b]]  \mod I_q^{n+2}(F)\ .$$
\end{thm}

The $p$-torsion part  of the Brauer group of $F$ is known to be isomorphic to $H^2(F)$ (see \cite[Section 9.2]{GilleSzamuely2006}).
The isomorphism is given by 
$$[\alpha,\beta)_{p,F} \mapsto \alpha \frac{d \beta}{\beta}\,.$$

The following lemma records certain  equalities for later use.

\begin{lem}\label{calcs} Take $a_1\ldots, a_n \in F^\times$  and  $b\in F\setminus\wp(F)$. 
Let $0\neq \beta = N_{F_b/F}(u)$ for some $u\in F_b$. 
\begin{enumerate}
\item[$(a)$]  For all $i=1,\ldots, n$ we have  $$b\frac{da_1}{a_1}\wg\ldots\wg \frac{da_n}{a_n} =(b+a_i)\frac{da_1}{a_1}\wg\ldots\wg \frac{da_n}{a_n}\mod d\Om^{n-1}(F)\,.$$ 
\item[$(b)$]  For all $i=1,\ldots, n$ we have in $H^{n+1}(F)$  $$b\frac{da_1}{a_1}\wg\ldots\wg \frac{da_n}{a_n} =b\frac{da_1}{a_1}\wg\ldots\wg\frac{d(a_i\beta)}{a_i\beta}\wg\ldots\wg \frac{da_n}{a_n}\,.$$ 
\item[$(c)$]   For all $i=1,\ldots, n$  we have in $H^{n+1}(F)$ $$b\frac{da_1}{a_1}\wg\ldots\wg \frac{da_n}{a_n} =(b+ a_i \beta)\frac{da_1}{a_1}\wg\ldots\wg\frac{d(a_i\beta)}{a_i\beta}\wg\ldots\wg \frac{da_n}{a_n}\,.$$ 
\end{enumerate}
\end{lem}
\begin{proof}
In all the  statements, it suffices to consider the case  $i=1$.  
Note  that as 
$$ d\left(b^{-1}\right)\wg db= d\left(\frac{b^{p-1}}{b^p} \right)\wg db = -\frac{ b^{p-2}}{b^p}db\wg db=0$$ for all $b\in F^\times$
we have 
$$d\left(b\frac{da_1}{a_1}\wg\ldots\wg \frac{da_n}{a_n}\right)=  db\wg\frac{da_1}{a_1}\wg\ldots\wg \frac{da_n}{a_n} \in d\Om^{n-1}(F)\,.$$

We first show $(a)$. We have that  
$$ a_1\frac{da_1}{a_1}\wg\ldots\wg \frac{da_n}{a_n}= da_1\wg \frac{da_2}{a_2}\wg\ldots\wg \frac{da_n}{a_n} = d\left({a_1} \frac{da_2}{a_2}\wg\ldots\wg \frac{da_n}{a_n}\right)\in d\Om^{n-1}(F)\,.$$
Hence the result follows from the additivity of $d$.

For $(b)$,   it suffices to consider the case $n=1$. In this case, the result follows from \cite[VII.1.9, (2)]{BO} via identifying cyclic $p$-algebras and symbols in $H^2(F)$.
Statement $(c)$  then follows immediately from $(a)$ and $(b)$.
\end{proof}

\section{Albert $p$-forms}\label{secal}
 
Let $F$ be a field of characteristic $p>0$. 
For $\alpha,\beta\in F$ and $\gamma,\delta\in F^\times$ the map $A(\alpha,\beta,\gamma,\delta): F_{\alpha+\beta} \oplus F_\alpha \oplus F_\beta \rightarrow F$
  given by 
$$(x,y,z)\mapsto N_{F_{\alpha+\beta}/F}(x)+\gamma N_{F_\alpha/F}(y)+\delta N_{F_\beta/F}(z)$$
is called an Albert $p$-form.
By the  pure part of the Albert $p$-form $A(\alpha,\beta,\gamma,\delta)$ we mean the restriction of $A(\alpha,\beta,\gamma,\delta)$ to 
$ F \oplus F_\alpha \oplus F_\beta \rightarrow F$.

\begin{rem}\label{Albertlinkage}
Note that for $p=2$ an Albert  $p$-form is   an   Albert form as defined in Section $2$. We also note the following:
\begin{enumerate}
\item If the Albert $p$-form above  has a non-trivial zero, then the cyclic algebras $[\alpha,\gamma)_{p,F}$ and $[\beta,\delta)_{p,F}$ are separably  linked. If $p=2$, then the converse also holds.
\item If the pure part of the Albert $p$-form above has a nontrivial zero, then the cyclic algebras $[\alpha,\gamma)_{p,F}$ and $[\beta,\delta)_{p,F}$ are inseparably linked. If $p=2$, then the converse also holds.
\end{enumerate}
\end{rem}

\begin{proof}
The `if'  statements follow immediately from \cite[Lemma 2.2]{Chapman:2017}. The converse statements for $p=2$ can be found in \cite{MammoneShapiro}.
\end{proof}

\begin{lem}\label{new}
Take  $\alpha\in F$ and  $\beta\in F^\times$. Then there exist $\alpha_1,\alpha_2\in F$ and $u\in F^\times$ such that  $\alpha= \alpha_1+\alpha_2$ and 
$$\alpha \frac{d \beta}{\beta}=  \alpha_1 \frac{d \beta}{\beta} =\alpha_2 \frac{d \beta u}{\beta  u} \in H^2(F)\,.$$
\end{lem}
\begin{proof}
If $\alpha\in \wp(F)$ then the result is trivial. Otherwise let $t= \frac{\alpha\beta-\alpha}{\beta}$, $\alpha_1 =\alpha+ \beta t^p$ and $\alpha_2= \alpha - \beta(t^p -t+\alpha)$. Then $\alpha=\alpha_1+\alpha_2$.
If $t=0$ then $\beta=1$ and again the result is trivial.
 If $t\neq0$ then  both $t^p$ and $u=-(t^p -t+\alpha)$ 	are norms of elements in the field $F_a$ (using $-1=(-1)^p$). Hence applying  \cref{calcs}, $(c)$ gives the result.
\end{proof}

\begin{thm}\label{h4=0}
Let $F$ be a field of characteristic $p$. If every Albert $p$-form over $F$ has a non-trivial zero then $H^4(F)=0$.
\end{thm}

\begin{proof} Let $\alpha\in F\setminus \wp(F)$,  $\beta,\gamma, \delta\in F^\times$ and 
 $\omega=\alpha \frac{d \beta}{\beta} \wedge \frac{d \gamma}{\gamma} \wedge \frac{d \delta}{\delta}\in H^4(F)$.
By  \cref{new} there exist $\alpha_1,\alpha_2\in F$ and $u\in F^\times$ such that  $\alpha= \alpha_1+\alpha_2$ and 
\begin{eqnarray}\label{alphas}\alpha \frac{d \beta}{\beta}=  \alpha_1 \frac{d \beta}{\beta} =\alpha_2 \frac{d \beta u}{\beta  u} \in H^2(F)\,.\end{eqnarray}
If $\alpha_2$ or $\alpha_1\in \wp(F)$, we have that $\omega=0\in  H^4(F)$. Therefore we may assume that $F_{\alpha_2}/F$ and $F_{\alpha_1}/F$ are non-trivial extensions. In particular the respective norm forms have no non-trivial zeros. 

By the hypothesis,  the Albert $p$-form 
$A(\alpha_1,\alpha_2,\gamma,\delta)$ has a non-trivial zero. That is, there exist $x\in F_\alpha$, $y\in F_{\alpha_1}$  and $z\in F_{\alpha_2}$  not all zero such that  
\begin{eqnarray}\label{2nd}N_{F_\alpha/F}(x)+\gamma N_{F_{\alpha_1}/F}(y)  + \delta N_{F_{\alpha_2}/F}(z)=0\,.\end{eqnarray}
If $x=0$, which holds if and only if $N_{F_a/F}(x)=0$, then 
$$\gamma N_{F_{\alpha_1}/F}(y) = -\delta N_{F_{\alpha_2}/F}(z) = \delta N_{F_{\alpha_2}/F}(-z)\neq 0 \,.$$ 
Fix $r=  N_{F_{\alpha_1}/F}(y)$ and $s= N_{F_{\alpha_2}/F}(-z)$. Then by (\ref{alphas}) and \cref{calcs}, $(c)$ we have 
\begin{eqnarray*}\omega &= & \alpha_1 \frac{d \beta}{\beta} \wedge \frac{d \gamma}{\gamma} \wedge \frac{d \delta}{\delta}
= \alpha_1 \frac{d \beta}{\beta} \wedge \frac{d \gamma r}{\gamma r} \wedge \frac{d \delta}{\delta}\\
&=&\alpha_2  \frac{d \beta u }{\beta u} \wedge \frac{d \gamma r}{\gamma r} \wedge \frac{d \delta}{\delta} = \alpha_2  \frac{d \beta u }{\beta u} \wedge \frac{d \gamma r}{\gamma r} \wedge \frac{d \delta s}{\delta s} = 0 \,,
\end{eqnarray*}
where the last equality follows from $\gamma r =\delta s$. 

Assume now that $x\neq 0$, and hence $N_{F_a/F}(x)\neq 0$. Let $\eta\in F_a$  be such that $N_{F_a/F}(\eta)=\alpha$. Fix $q= N_{F_a/F}(\eta x^{-1})$. Multiplying (\ref{2nd}) by $q$ and using the multiplicativity of the norm form gives 
$ \alpha +q\gamma N_{F_{\alpha_1}/F}(y)  + q\delta N_{F_{\alpha_2}/F}(z)=0 \,.$
As $q$ is also a norm of an element in $F_\alpha$, \cref{calcs}, $(c)$ gives
$$\omega = \alpha \frac{d \beta}{\beta} \wedge \frac{d q \gamma}{q\gamma} \wedge \frac{d q\delta}{q\delta}\,.$$
Hence we may assume that $\omega =\alpha \frac{d \beta}{\beta} \wedge \frac{d \gamma}{\gamma} \wedge \frac{d \delta}{\delta}$  and that $\alpha +\gamma N_{F_{\alpha_1}/F}(y)  + \delta N_{F_{\alpha_2}/F}(z)=0$. 
Fix 
$$  r= \left\{ \begin{array}{cc} N_{F_{\alpha_1}}(y) & \textrm{ if  }  y\neq0 \\ 
1  & \textrm{ otherwise}  \end{array} \right.
\quad \textrm{
and} 
\quad  s= \left\{ \begin{array}{cc} N_{F_{\alpha_2}}(z) & \textrm{ if  }  z\neq0 \\ 
1  & \textrm{ otherwise}  \end{array} \right.\,.
$$
Then by (\ref{alphas}) and \cref{calcs}, $(c)$ we have for some $u\in F^\times$  
\begin{eqnarray*}\omega &= & \alpha_1 \frac{d \beta}{\beta} \wedge \frac{d \gamma}{\gamma} \wedge \frac{d \delta}{\delta}
= \alpha_1 \frac{d \beta}{\beta} \wedge \frac{d \gamma r}{\gamma r} \wedge \frac{d \delta}{\delta}
=\alpha_2  \frac{d \beta u }{\beta u} \wedge \frac{d \gamma r}{\gamma r} \wedge \frac{d \delta}{\delta}
\\ & = &\alpha_2  \frac{d \beta u }{\beta u} \wedge \frac{d \gamma r}{\gamma r} \wedge \frac{d \delta s}{\delta s} = \alpha  \frac{d \beta  }{\beta } \wedge \frac{d \gamma r}{\gamma r} \wedge \frac{d \delta s}{\delta s}
 \\&=& (\alpha +\gamma N_{F_{\alpha_1}/F}(y)  + \delta N_{F_{\alpha_2}/F}(z) )  \frac{d \beta  }{\beta } \wedge \frac{d \gamma r}{\gamma r} \wedge \frac{d \delta s}{\delta s} = 0 \,.
\end{eqnarray*}\end{proof}

\begin{thm}\label{pure}
Let $F$ be a field of characteristic $p$.
Suppose that for all $\alpha\in F$ and $\gamma,\delta\in F^\times$, the pure part of the Albert $p$-form $A(\alpha,\alpha,\gamma,\delta)$ has a non-trivial zero. Then $H^3(F)=0$.
\end{thm}
\begin{proof} Let $\alpha,\gamma,\delta\in F^\times$ and 
 $\omega=\alpha \frac{d \gamma}{\gamma} \wedge \frac{d \delta}{\delta}\in H^3(F)$. We may assume that $\alpha\notin \wp(F)$ and hence that the field extension $F_\alpha/F$ is non-trivial.
Consider the Albert $p$-form $A(\alpha,\alpha, \gamma, \delta)$. 
By the hypothesis, there exist $x\in F$ and $y,z\in F_\alpha$ not all zero such that
$$x^p + \gamma N_{F_\alpha/F}(y)+ \delta N_{F_\alpha/F}(z)=0\,.$$
Clearly at least one of $y$ or $z$ must be non-zero. Fix 
$$  r= \left\{ \begin{array}{cc} N_{F_\alpha/F}(y) & \textrm{ if  } y\neq0 \\ 
1  & \textrm{ otherwise}  \end{array} \right.
\quad 
\textrm{and }
\quad  s= \left\{ \begin{array}{cc} N_{F_\alpha/F}(-z) & \textrm{ if  }  z\neq0 \\ 
1  & \textrm{ otherwise}  \end{array} \right.\,.
$$
Then by \cref{calcs}, $(b)$ we have  
$$\omega=\alpha \frac{d\gamma r}{\gamma r} \wedge \frac{d\delta s}{\delta s}\,.$$ As $\gamma r$ and $\delta s$ differ by an element in $F^p$ we have  
$\frac{d \gamma r}{\gamma r} \wedge \frac{d \delta s}{\delta s}=0$
and hence $\omega=0$.
\end{proof}

\begin{rem}
Over fields of characteristic $2$, the pure parts of Albert $2$-forms
of the type  $A(\alpha,\alpha,\gamma,\delta)$ 
  correspond  (up to scaling) to $5$-dimensional   Pfister neighbours (see \cite[(23.10)]{EKM})
and hence it is clear that  if all these forms are isotropic, then $I_q^3F$ is trivial.
\end{rem}

By  Remark \ref{Albertlinkage} and  \cref{h4=0} have that $H^4(F)=0$ for a linked field $F$ of characteristic $2$.

\begin{ques}\label{h4=0?}
Let $F$ be a field of characteristic $p>2$. 
If every two cyclic algebras of degree $p$ over $F$ are separably linked, is
 $H^4(F)=0$? If every two such algebras are inseparably linked, is $H^3(F)=0$?
\end{ques}

\section{Linked Fields of characteristic $2$}

The following result was shown in {\cite{Faivre:thesis}}.

\begin{thm}[{\cite[Theorem 3.3.10]{Faivre:thesis}}]\label{faivre}
If $F$ is a linked field with $\operatorname{char}(F)=2$ and $I_q^4 F=0$ then the possible values its $u$-invariant can take are $0,2,4$ and $8$.
\end{thm}

Since a field of characteristic $2$ being linked is equivalent to every Albert form over $F$ being isotropic by Remark \ref{Albertlinkage}, combining \cref{faivre}, \cref{h4=0} and \cref{kato} gives the following result:

\begin{cor}\label{linkeduinv}
If $F$ is a linked field with $\operatorname{char}(F)=2$ then the possible values its $u$-invariant can take are $0,2,4$ and $8$.
\end{cor}

\cref{linkeduinv}
follows more directly from \cref{faivre} if one can show that every $4$-fold Pfister form contains an Albert form as a subform, as then clearly $I_q^4(F)=0$ if the field is linked. 
As this result is also of independent interest, we give a proof below. The computations are similar to those used in \cref{h4=0}. 
We use the following well-known isometry. This can be derived  from,  for example, \cite[(2.4) and (2.6)]{dolphin:conic}. It can also be directly derived from 
\cref{calcs} and \cref{kato}.

\begin{lem}\label{iso} Assume $\operatorname{char}(F)=2$.
Let  $b\in F^\times$, $a\in F$, $x,y\in F$ not both zero and $\beta=x^2+xy+ay^2$. Then we have 
$ \langle \langle b,a]]\simeq \langle \langle b\beta, a+b\beta]]\,.$
\end{lem}

\begin{lem}\label{suballem} Assume $\operatorname{char}(F)=2$.
Let $\pi$ be a $2$-fold Pfister form over $F$ and $\lambda_i\in F^\times$ for $i=1,2,3$. The the form 
$\rho=\qf{ \lambda_1,\lambda_2,\lambda_3}\otimes \pi
$
contains an Albert subform.
\end{lem}
\begin{proof}
Let $\alpha\in F$ and $\beta\in F^\times$ such that $\pi= \langle \langle \beta,\alpha]]$.
Fix $t= \frac{\alpha+\beta \alpha}{\beta}$, $\alpha_1= \alpha + \beta t^2$ and $\alpha_2= \alpha +\beta(t^2+t +\alpha)$. Then $\alpha=\alpha_1+\alpha_2$. 
%
%
%
%
%
%
%
Using \cref{iso}, we see that the forms $[1,\alpha]$, $[1,\alpha_1]$ and $[1,\alpha_2]$ are all subforms of $\pi$.
Consequently, the form $$\psi=\lambda_1[1,\alpha] \perp \lambda_2 [1,\alpha_1] \perp \lambda_3[1,\alpha_2]$$ is a subform of $\rho$. The Arf invariant of this form is $\alpha+\alpha_1+\alpha_2=0$.
Therefore $\psi$ is an Albert form.
\end{proof}

\begin{cor}\label{subalbert}
Given a field $F$ with $\operatorname{char}(F)=2$, every  $4$-fold Pfister form contains an Albert subform.
\end{cor}
\begin{proof}
Let  $\varphi=\langle \langle \delta,\gamma,\beta,\alpha]]$ be a  4-fold Pfister form over $F$ and let $\pi =\langle \langle \beta,\alpha]]$. Then $\rho= \pi\perp \delta\pi\perp \gamma\pi$
is a subform of $\varphi$. The form $\rho$, and hence $\varphi$, contains an Albert subform by \cref{suballem}.
\end{proof}

\begin{thm}\label{newElmanLam}
If $F$ is a linked field with $\operatorname{char}(F)=2$ then $I_q^4 F=0$.
\end{thm}

\begin{proof}
Let $\varphi$ be a  4-fold Pfister form over $F$.
By   \cref{subalbert} it contains an Albert subform $\rho$.
Since $F$ is linked, $\rho$ must be isotropic.
Therefore $\varphi$ is hyperbolic.
\end{proof}

\begin{rem}\label{charnot2}
A result analogous to  \cref{suballem}  holds  for those fields $F$ with $\operatorname{char}(F) \neq 2$ containing a square root of $-1$.
Let $\pi=\langle \langle \alpha,\beta \rangle \rangle$ be a 2-fold Pfister form over $F$ and let $\lambda_1,\lambda_2,\lambda_3$ be three arbitrary elements in $F^\times$.
Then $\rho=\langle \lambda_1,\lambda_2,\lambda_3 \rangle \otimes \pi$ contains the subform $$\lambda_1 \langle \alpha,\beta \rangle \perp \lambda_2 \langle \alpha,\alpha \beta \rangle \perp \lambda_3 \langle \alpha \beta,\beta \rangle\,.$$ This has  trivial discriminant, and so it is an Albert form.
Hence, every $4$-fold Pfister form over $F$ contains an Albert form. 

This is not the case  in general for fields that do not contain a square root of $-1$.  
For example, the unique anisotropic $4$-fold Pfister form $\langle \langle -1,-1,-1,-1 \rangle \rangle$ over $\mathbb{R}$  clearly has no Albert subform, as all Albert forms over $\mathbb{R}$ are isotropic.
\end{rem}

\section*{Acknowledgements} 
We  thank the referee for  very useful suggestions. In particular pointing out that our results hold for a broader definition of Albert  $p$-forms than used in a preliminary version and  several resulting  simplifications to  the arguments in Section \ref{secal}.

The second author was supported by   
 \emph{Automorphism groups of locally finite trees} (G011012) with the Research Foundation, Flanders, Belgium (F.W.O.~Vlaanderen). %

\section*{Bibliography}
\bibliographystyle{amsalpha}

\end{document}